\documentclass[10pt]{amsart}
\usepackage{amssymb,amscd,latexsym,paralist}
\usepackage[all]{xy}
\newcommand{\F}{{\mathbb{F}}}
\newcommand{\Z}{{\mathbb{Z}}}
\renewcommand{\mod}{\;\mathrm{mod}\;}
\newcommand{\Imm}{\mathrm{Im}\,}
\newcommand{\ea}{\mathfrak{a}}
\newcommand{\eb}{\mathfrak{b}}
\newcommand{\silo}{\xrightarrow{\sim}}
\newcommand{\verk}{\mbox{\scriptsize $\,\circ\,$}}
\newtheorem{theorem}{Theorem}
\newtheorem{lemma}[theorem]{Lemma}
\newtheorem{proposition}[theorem]{Proposition}
\newtheorem{cor}[theorem]{Corollary}
\newtheorem{remark}[theorem]{Remark}
\newenvironment{rem}{\noindent {\bf Remark.}}{}
\newenvironment{rems}{\noindent {\bf Remarks.}}{}
\newcommand{\beweisende}{\hspace*{\fill} $\Box$}
\begin{document}
\title{An alternative to Witt vectors}
\author{Joachim Cuntz and Christopher Deninger}
\address{J. Cuntz, Mathematisches Institut, Einsteinstr.62, 48149 M\"unster, Germany\newline
C. Deninger, Mathematisches Institut, Einsteinstr. 62, 48149 M\"unster, Germany}
\thanks{Research supported by DFG through CRC 878 and by ERC through AdG 267079}
\begin{abstract}\noindent
The ring of Witt vectors associated to a ring R is a classical tool in algebra. We introduce a ring $C(R)$ which is more easily constructed and which is isomorphic to the Witt ring $W(R)$ for a perfect $\F_p$-algebra $R$. It is obtained as the completion of the monoid ring $\Z R$, for the multiplicative monoid $R$, with respect to the powers of the kernel of the natural map $\Z R\to R$.

\end{abstract}
\dedicatory{Dedicated to our friend and colleague Peter Schneider}
\maketitle

\section{Introduction}
Since the work of Witt on discretely valued fields with given perfect residue field in \cite{W} the ``vectors'' that carry his name have become important in many branches of mathematics. In \cite{L1}, \cite{L2} Lazard gave a new approach to Witt vectors generalizing the theory to the case of a perfect $\F_p$-residue algebra. This approach is the one used in Serre \cite{S} for example. Addition and multiplication of Witt vectors are defined by certain universal polynomials. This description is cumbersome. While thinking about periodic cyclic cohomology for $\F_p$-algebras we found an alternative $C (R)$ to the $p$-typical Witt ring $W (R)$ of a perfect $\F_p$-algebra $R$. The rings $C (R)$ and $W (R)$ are canonically isomorphic but the construction of $C (R)$ as a completion (hence the name) of a monoid algebra $\Z R$ is much simpler. We have therefore made an effort to develop the properties of $C (R)$ independently of the theory of $W (R)$. 

Using the approach in \cite{Cu-Qu1}, \cite{Cu-Qu2} we can define periodic 
cyclic homology for a ring $R$ using completed extensions by free 
(noncommutative) $\Z$-algebras. If one applies this procedure to an 
$\F_p$-algebra $R$, the completion $C(R)$ of the free $\Z$-module $\Z R$ 
appears as a natural intermediate step. 

If the $\F_p$-algebra $R$ is not perfect, $C(R)$ is still defined and different from $W (R)$. However a somewhat more involved construction in the same spirit does give $W (R)$ in general. We will address this in a subsequent paper together with applications.

In this note all rings are commutative with $1$ and all ring homomorphisms map $1$ to $1$. The background reference is \cite{S} II \S\,4--\S\,6.
\section{Construction and properties of $C (R)$}
A $p$-ring $A$ is a commutative ring with unit which is Hausdorff and complete for the topology defined by a sequence of ideals $\ea_1 \supset \ea_2 \supset \ldots$ with the following properties:\\
{1)} $\ea_i \ea_j \subset \ea_{i+j}$ for $i,j \ge 1$\\
{2)} $A / \ea_1 = R$ is a perfect $\F_p$-algebra, i.e. the Frobenius homomorphism $x \mapsto x^p$ is an isomorphism of $R$.

For a $p$-ring we have $p \in \ea_1$ and hence $\ea_i \supset p^iA$. A $p$-ring is called strict if $\ea_i = p^iA$ and if $p$ is not a zero divisor in $A$. It is known that for every perfect $\F_p$-algebra $R$ there is a strict $p$-ring $A = W (R)$ with $A / p A = R$. The pair $(W (R) , W (R) \to R)$ is unique up to a unique isomorphism.


View $R$ as a monoid under multiplication  and let $\Z R$ be the monoid algebra of $(R , \cdot)$. Its elements are formal sums of the form $\sum_{r \in R} n_r [r]$ with almost all $n_r = 0$. Addition and multiplication are the obvious ones. Note that $[1] = 1$ but $[0] \neq 0$. Multiplicative maps $R \to B$ into commutative rings mapping $1$ to $1$ correspond to ring homomorphisms $\Z R \to B$. The identity map $R = R$ induces the surjective ring homomorphism $\pi : \Z R \to R$ which sends $\sum n_r [r]$ to $\sum n_r r$. Let $I$ be its kernel, so that we have an exact sequence
\[
0 \longrightarrow I \longrightarrow \Z R \xrightarrow{\pi} R \longrightarrow 0 \; .
\]
It is not difficult to see that as a $\Z$-module $I$ is generated by elements of the form $[r] + [s] - [r+s]$ for $r,s \in R$. We will not use this fact in the sequel.
The multiplicative isomorphism $r \mapsto r^p$ of $R$ induces a ring isomorphism $F : \Z R \to \Z R$ mapping $\sum n_r [r]$ to $\sum n_r [r^p]$. It satisfies $F (I) = I$.

Let $C (R) = \varprojlim_{\nu} \Z R / I^{\nu}$ be the $I$-adic completion of $\Z R$. By construction $C (R)$ is Hausdorff and complete for the topology defined by the ideals $\ea_i = \widehat{I^i}$ where
\[
\widehat{I^i} = \varprojlim_{\nu} I^i / I^{\nu} \subset C (R)\; .
\]
Note that at this stage we do not know that $\widehat{I^i} = \hat{I}^i$ since $\Z R$ is not Noetherian in general. Condition {1)} above is satisfied and {2)} as well since
\[
C (R) / \hat{I} = \Z R / I = R \; .
\]
Hence $C (R)$ is a $p$-ring. The construction of $C (R)$ is functorial in $R$.

\begin{theorem}
\label{t1}
Let $R$ be a perfect $\F_p$-algebra. Then $C (R)$ is a strict $p$-ring with $C (R) / p C (R) = R$.
\end{theorem}

The result is an immediate consequence of the universal properties shared by $C (R)$ and any strict $p$-ring with residue algebra $R$, once one knows that such a strict $p$-ring exists; see remark \ref{t5} below. In the following we give a self-contained proof of theorem \ref{t1} which does not use this information.

Consider the ``arithmetic derivation'' $\delta : \Z R \to \Z R$ defined by the formula
\[
\delta (x) = \frac{1}{p} (F (x) - x^{p}) \; .
\]
It is well defined since $F (x) \equiv x^{p} \mod p \Z R$ and since $\Z R$ being a free $\Z$-module has no $\Z$-torsion. The following formulas for $x,y \in \Z R$ are immediate
\begin{equation}
\label{eq:1}
\delta (x+y) = \delta (x) + \delta (y) - \sum^{p-1}_{\nu=1} \frac{1}{p} {p \choose \nu} x^{\nu} y^{p-\nu}
\end{equation}
and
\begin{equation}
\label{eq:2}
\delta (xy) = \delta (x) F (y) + x^{p} \delta (y) \; .
\end{equation}
Applying \eqref{eq:2} inductively gives the relation
\begin{equation}
\label{eq:3}
\delta (x_1 \cdots x_n) = \sum^n_{\nu =1} x^{p}_1 \cdots x^{p}_{\nu-1} \delta (x_{\nu}) F (x_{\nu + 1}) \cdots F (x_n) \quad \text{for} \; x_i \in \Z R \; .
\end{equation}
Equation \eqref{eq:1} shows that we have
\begin{equation}
\label{eq:4}
\delta (x+y) \equiv \delta (x) + \delta (y) \mod I^n \quad \text{if $x$ or $y$ is in $I^n$} \; .
\end{equation}
Together with \eqref{eq:3} it follows that
\begin{equation}
\label{eq:5}
\delta (I^n) \subset I^{n-1} \quad \text{for} \; n \ge 1 \; .
\end{equation}

\begin{lemma}
\label{t2}
Let $R$ be a perfect $\F_p$-algebra and $n \ge 1$ an integer. \\
a) If $pa \in I^n$ for some $a \in \Z R$ then $a \in I^{n-1}$.\\
b) $I^n = I^{\nu} + p^n \Z R$ for any $\nu \ge n$.
\end{lemma}

\begin{proof}
a) According to formula \eqref{eq:5} we have $\delta (pa) \in I^{n-1}$. On the other hand, by definition:
\[
\delta (pa) = F (a) - p^{p-1} a^{p} \; ,
\]
and therefore since $pa \in I^n$
\[
\delta (pa) \equiv F (a) \mod I^{n} \; .
\]
It follows that $F (a) \in I^{n-1}$ and hence $a \in I^{n-1}$ since $F$ is an automorphism with $F (I) = I$.\\
b) We prove the inclusion $I^n \subset I^{\nu} + p^n \Z R$ for $\nu \ge n$ by induction with respect to $n \ge 1$. The other inclusion is clear. For $y \in \Z R$ and $\nu \ge 1$ we have
\[
F^{\nu} (y) \equiv y^{p^{\nu}} \mod p \Z R \; .
\]
Applying this to $y = F^{-\nu} (x)$, we get for all $x \in \Z R$
\[
x \equiv F^{-\nu} (x)^{p^{\nu}} \mod p \Z R \; .
\]
For $x \in I$ this shows that $x \in I^{\nu} + p \Z R$ settling the case $n = 1$ of the assertion.\\
Now assume that $I^n \subset I^{\nu} + p^n \Z R$ has been shown for a given $n \ge 1$ and all $\nu \ge n$. Fix some $\nu \ge n+1$ and consider an element $x \in I^{n+1}$. By the induction assumption $x = y + p^n z$ with $y \in I^{\nu}$ and $z \in \Z R$. Hence $p^n z = x-y \in I^{n+1}$. Using assertion a) of the lemma repeatedly shows that $z \in I$. Hence $z \in I^{\nu} + p \Z R$ by the case $n = 1$. Writing $z = a + p b$ with $a \in I^{\nu}$ and $b \in \Z R$ we find
\[
x = (y + p^n a) + p^{n+1} b \in I^{\nu} + p^{n+1} \Z R \; .
\]
Thus we have shown the induction step $I^{n+1} \subset I^{\nu} + p^{n+1} \Z R$.
\end{proof}

After these preparations the {\it proof of theorem \ref{t1}} follows easily: We have to show that $p^n C (R) = \widehat{I^n}$ for all $n \ge 1$ and that $p$ is not a zero divisor in $C (R)$. Let $p^{-n} (I^{\nu})$ be the inverse image of $I^{\nu}$ under $p^n$-multiplication on $\Z R$. Then for any $\nu \ge n \ge 1$ we have an exact sequence where the surjectivity on the right is due to part b) of Lemma \ref{t2}:
\[
0 \longrightarrow p^{-n} (I^{\nu}) / I^{\nu} \longrightarrow \Z R / I^{\nu} \xrightarrow{p^n} I^n / I^{\nu} \longrightarrow 0 \; .
\]
From this we get an exact sequence of projective systems whose transition maps for $\nu \ge n$ are the reduction maps. Set $N_{\nu} = p^{-n} (I^{\nu}) / I^{\nu}$. Then we have an exact sequence
\[
0 \longrightarrow \varprojlim_{\nu} N_{\nu} \longrightarrow C (R) \xrightarrow{p^n} \widehat{I^n} \longrightarrow \varprojlim_{\nu}\,\!^{(1)} N_{\nu} \; .
\]
The transition map $N_{\nu + n} \to N_{\nu}$ is the zero map since $a \in p^{-n} (I^{\nu +n })$ implies $p^n a \in I^{\nu + n}$ and hence $a \in I^{\nu}$ by part a) of Lemma \ref{t2}. In particular $(N_{\nu})$ is Mittag--Leffler, so that $\varprojlim_{\nu}^{(1)} (N_{\nu}) = 0$. It is also clear now that $\varprojlim_{\nu} N_{\nu} = 0$. It follows that $p^n$-multiplication on $C (R)$ is injective with image $\widehat{I^n}$. Hence $p$ is not a zero-divisor and $\widehat{I^n} = p^n C (R)$. \beweisende

\begin{rems} {\bf 1)} If $R$ is a perfect $\F_p$-algebra there is an isomorphism
\[
R \silo I^n / I^{n+1} \quad \text{given by} \; r \longmapsto p^n [r] \; .
\]
This follows because:
\[
I^n / I^{n+1} = \widehat{I^n} / \widehat{I^{n+1}} = p^n C(R) / p^{n+1} C(R) \overset{p^{-n}}{=} C (R) / p C (R) = R \; .
\]
{\bf 2)} The automorphism $F$ of $\Z R$ satisfies $F (I) = I$. Hence it induces an automorphism $F$ of $C (R)$ which lifts the Frobenius automorphism of the perfect $\F_p$-algebra $R$. The Verschiebung $V : C (R) \to C (R)$ is the additive homomorphism defined by $V (x) = p F^{-1} (x)$. By definition $\Imm V^i = p^i C (R)$ and $V \verk F = F \verk V = p$. The projection $\pi : C (R) \to R$ has a multiplicative splitting defined as the composition $\omega : R \hookrightarrow \Z R \to C (R)$. Frobenius $F$, Verschiebung $V$ and Teichm\"uller lift $\omega$ are well known extra structures on rings of Witt vectors.
\end{rems}

\begin{proposition}
\label{t3}
If $R = K$ is a perfect field of characteristic $p$ then $C (K)$ is a discrete valuation ring of mixed characteristic with residue field $K$.
\end{proposition}

\begin{proof}
This is true for any strict $p$-ring $W$ with residue field $K$. The well known argument is as follows. By assumption $pW$ is a maximal ideal of $W$. For $x \in W \setminus pW$ choose $y \in W$ with $xy \equiv 1 \mod p$. Then $(xy)^{p^{\nu}} \equiv 1 \mod p^{\nu}$ by \cite{S}, II, \S\,4, Lemma 1. Hence $x \mod p^{\nu}$ is a unit in $W / p^{\nu} W$ and therefore $x = (x \mod p^{\nu})_{\nu \ge 0}$ is a unit in $W$. Hence the ring $W$ is local with unique maximal ideal $pW$. Since $W$ is separated i.e. $\bigcap^{\infty}_{\nu = 1} p^{\nu} W = 0$ it follows that for every $0 \neq a \in W$ there is a unique integer $v (a) \ge 0$ with $a = p^{v (a)} x$ and $x \in W \setminus p W$ i.e. $x \in W^*$. Since multiplication with $p$ is injective on $W$, it follows that $W$ is an integral domain. The map $v : W \setminus \{ 0 \} \to \Z$ satisfies $v (ab) = v (a) + v (b)$ by definition and $v (a+b) \ge \min (v (a) , v (b))$ because as seen above, an element of $W$ is a unit if and only if its reduction $\mod p$ is non-zero. The valuation $v$ extends uniquely to a discrete valuation on the quotient field $Q$ of $W$ with valuation ring $W$.
\end{proof}

\begin{rem}
In particular $C (K)$ is Noetherian while in general $\Z K$ is very far from being Noetherian.
\end{rem}
\medskip

As a topological additive group, $C (R)$ has another description which is sometimes useful. Let $\mathfrak{b}$ be a basis of the $\F_p$-algebra $R$ and let $\Z \eb$ be the free $\Z$-module with basis $\eb$. The inclusion $\eb \subset R$ induces an additive homomorphism
\[
\Z \eb \hookrightarrow \Z R \longrightarrow C (R)
\]
and hence a map
\[
\widehat{\Z \eb} = \varprojlim_{n} \Z \eb / p^n \Z \eb \longrightarrow C (R) \; .
\]

\begin{proposition} \label{t3a}
If $R$ is a perfect $\F_p$-algebra, the map $\widehat{\Z \eb} \to C (R)$ is a topological isomorphism of additive groups. In particular any inclusion $R_1 \hookrightarrow R_2$ resp. surjection $R_1 \twoheadrightarrow R_2$ of perfect $\F_p$-algebras induces an inclusion $C (R_1) \hookrightarrow C (R_2)$ resp. surjection $C (R_1) \twoheadrightarrow C (R_2)$ with a continuous additive splitting.
\end{proposition}

\begin{proof}
By theorem \ref{t1} we have to show that for each $n \ge 1$ the additive map 
\[
\alpha_n : \Z \eb / p^n \Z \eb \longrightarrow C (R) / p^n C (R)
\]
is an isomorphism. For $n = 1$ this is true because $\F_p \eb = R$ since $\eb$ is an $\F_p$-basis for $R$. Now assume that $\alpha_n$ is an isomorphism and consider the commutative diagram
\[
\xymatrix{
0 \ar[r] & \Z \eb / p \Z \eb \ar[d]^{\alpha_1} \ar[r]^{p^n} & \Z \eb / p^{n+1} \Z \eb \ar[d]^{\alpha_{n+1}} \ar[r] & \Z \eb / p^n \Z \eb \ar[d]^{\alpha_n} \ar[r] & 0 \\
0 \ar[r] & C (R) / p C (R) \ar[r]^{p^n} & C (R) / p^{n+1} C (R) \ar[r] & C (R) / p^n C (R) \ar[r] & 0
}
\]
The upper sequence is exact and because of theorem \ref{t1} the lower sequence is exact as well. Hence $\alpha_{n+1}$ is an isomorphism. The remaining assertions follow immediately.
\end{proof}

\begin{rem}
If the basis $\eb$ happens to be closed under multiplication then $\Z \eb$ is a ring and $\widehat{\Z \eb} \to C (R)$ an isomorphism of rings. This is the case in the following example. The perfect $\F_p$-algebra $R = \F_p [t^{p^{-\infty}}_1 , \ldots , t^{p^{-\infty}}_d]$ has a basis $\eb$ consisting of monomials. This basis is multiplicatively closed and hence $C (R)$ is the $p$-adic completion of the monoid algebra $\Z \eb$ i.e. of the algebra $\Z [t^{p^{-\infty}}_1 , \ldots , t^{p^{-\infty}}_d]$.
\end{rem}

\begin{proposition}
\label{t4}
Let $A$ be a $p$-ring with perfect residue algebra $R$ as above. Then there is a unique homomorphism of rings $\hat{\alpha} : C (R) \to A$ such that the following diagram commutes:
\begin{equation}
\label{eq:9}
\xymatrix{
C (R) \ar[rr]^{\hat{\alpha}} \ar[dr]_{\pi} & & A \ar[dl]^{\pi_A} \\
 & R &
}
\end{equation}
\end{proposition}

\begin{rem}
This is true for any strict $p$-ring instead of $C (R)$, c.f. \cite{S}, II, \S\,5, Proposition 10. However in our case the argument is particularly simple and we do not even have to know that $C (R)$ is strict.
\end{rem}

\begin{proof}
Since $A$ is a $p$-ring, there is a unique multiplicative section $\alpha_0 : R \to A$ of $\pi_A$, c.f. \cite{S}, II, \S\,4, Proposition 8.
Hence there is a unique ring homomorphism $\alpha : \Z R \to A$ such that the diagram
\[
\xymatrix{
\Z R \ar[rr]^{\alpha} \ar[dr]_{\pi} & & A \ar[dl]^{\pi_A} \\
 & R &
}
\]
commutes. Since $\alpha (I) \subset \ea_1$ we have $\alpha (I^{\nu}) \subset \ea^{\nu}_1 \subset \ea_{\nu}$ and therefore $\alpha$ extends to a unique and automatically continuous homomorphism $\hat{\alpha} : C (R) \to A$ such that \eqref{eq:9} commutes.
\end{proof}

\begin{remark}
\label{t5}
\em As we saw above it is immediate that $C (R)$ is a $p$-ring with residue algebra $R$. Showing directly that $C (R)$ is a strict $p$-ring required some thought. If one already knows that there is a strict $p$-ring $W$ with residue algebra $R$, then it is easy to see that $C (R)$ is isomorphic to $W$ and hence strict. Here is the argument:\\
The universal property of strict $p$-rings \cite{S} II \S\,5 Proposition 10 gives us a unique homomorphism $\beta : W \to C (R)$ such that the diagram
\[
\xymatrix{
W \ar[rr]^{\beta} \ar[dr] & & C (R) \ar[dl]^{\pi} \\
& R &
}
\]
commutes. On the other hand by proposition \ref{t4} there is a unique homomorphism $\hat{\alpha} : C (R) \to W$ such that
\[
\xymatrix{
C (R) \ar[rr]^{\hat{\alpha}} \ar[dr]_{\pi} & & W \ar[dl] \\
& R &
}
\]
commutes. The map $\alpha \verk \beta$ is the identity on $W$ because of the universal property for the strict $p$-ring $W$. The map $\beta \verk \alpha$ is the identity on $C (R)$ by proposition \ref{t4} because $C (R)$ is a $p$-ring. It follows that $C (R) \cong W$ is a strict $p$-ring. An equally simple proof may be given by using the characterization of the triple $(W (R) , R \hookrightarrow W (R) , \pi : W (R) \to R)$ in \cite{CC} Proposition 3.1 which is based on \cite{F} Theorem 1.2.1. 
\end{remark}

From the preceeding remark we get the following corollary:

\begin{cor} \label{t6}
Let $W_n (R)$ be the truncated ($p$-typical) Witt ring of the perfect $\F_p$-algebra $R$. There is a unique homomorphism of rings $\alpha_n : \Z R / I^n \to W_n (R)$ inducing the standard multiplicative embedding $R \hookrightarrow W_n (R)$ and making the following diagram commute
\[
\xymatrix{
\Z R / I^n \ar[rr]^{\overset{\scriptstyle \alpha_n}{\sim}} \ar[dr] & & W_n (R) \ar[dl]\\
 & R &
}
\]
Moreover, $\alpha_n$ is an isomorphism.
\end{cor}

\begin{proof}
Let $W (R)$ be the $p$-typical Witt ring of $R$. According to remark \ref{t5} there is a commutative diagram
\[
\xymatrix{
C (R) \ar[rr]^{\overset{\scriptstyle \hat{\alpha}}{\sim}} \ar[dr] & & W (R) \ar[dl]\\
 & R &
}
\]
Reducing  $\mod p^n$ and noting that $W (R) / p^n W (R) = W_n (R)$ and
\[
C (R) / p^n C (R) = C (R) / \widehat{I^n} = \Z R / I^n
\]
we get an isomorphism $\alpha_n$ as desired. There is a unique ring homomorphism $\alpha : \Z R \to W_n (R)$ prolonging the multiplicative embedding $R \hookrightarrow W_n (R)$. Hence $\alpha_n$ is uniquely determined.
\end{proof}

As a set $W_n (R)$ is $R^n$. Addition and multiplication are given by certain universal polynomials in $2n$ variables over $\Z$. We now describe the isomorphism $\alpha_2$. Note that \eqref{eq:4} and \eqref{eq:5} imply that $\delta$ induces a (non-additive) map
\[
\overline{\delta} : \Z R / I^2 \longrightarrow \Z R / I = R \; .
\]
We also have the ring homomorphism of reduction $\pi : \Z R / I^2 \to \Z R / I = R$.

\begin{proposition}
\label{t7}
The isomorphism
\[
\alpha_2 : \Z R / I^2 \xrightarrow{\sim} W_2 (R) = R^2
\]
is given by the map $(\pi , \overline{\delta})$.
\end{proposition}

\begin{proof}
The composition $R \to \Z R / I^2 \to W_2 (R) = R^2$ is the standard multiplicative embedding. One checks that $\alpha_2$ is a ring homomorphism using the formulas for addition and multiplication on $W_2 (R) = R^2$:
\[
(x,y) + (x', y') = \Big( x + x' , y + y' - \frac{1}{p} \sum^{p-1}_{\nu = 1} {p \choose \nu} x^{\nu} x'^{p-\nu} \Big)
\]
and
\[
(x,y) \cdot (x' , y') = (xx' , x^{'p} y + y' x^p + pyy') \; .
\]
Using corollary \ref{t6} the assertion follows.
\end{proof}
\begin{rem}
With respect to the ordinary $R$-module structure on $R^2$ the map $\alpha_2$ is non-linear. Hence the simple addition and multiplication on $\Z R / I^2$ become something non-obvious on $R^2$. We have not tried to describe $\alpha_n$ for $n \ge 3$ by explicit formulas. 
\end{rem}
\bigskip

It is interesting to compare the $I$-adic completion $C(R)$ of $\Z
R$ with its $p$-adic completion i.e. the completion with respect to
powers of the ideal $p\Z R$. Lemma 2 b) shows that the projective
system $(I^n/p^n\Z R)_n$ satisfies the Mittag-Leffler condition.
Therefore we obtain the following exact sequence
\begin{equation}\label{xyz} 0\quad\to\quad\mathop{\lim}\limits_{
{\scriptstyle\longleftarrow} } I^n/p^n\Z R \quad \to
\quad\mathop{\lim}\limits_{ {\scriptstyle\longleftarrow} } \Z
R/p^n\Z R \quad \to \quad\mathop{\lim}\limits_{
{\scriptstyle\longleftarrow} } \Z R/I^n \quad \to \quad
0
\end{equation}
which describes the kernel of the natural map from the $p$-adic
completion to $C(R)$.

Now, if $R$ is a finite perfect $\F _p$-algebra, then the $p$-adic
completion of $\Z R$ is $\Z_p R$ the monoid algebra of $R$ over $\Z_p$ and $C(R)$ has an
instructive description as a complete subring of $\Z_p R$.

\begin{proposition} \label{t8}
Assume that $R$ is a finite perfect $\F _p$-algebra. Then there is
an idempotent $e$ in $\Z_p R$ such that $e(\Z_pR)$ is the kernel of
the natural map $\Z_pR\to C(R)$ and such that $(1-e)\Z_pR$ is
topologically isomorphic to $C(R)$.
\end{proposition}
\begin{proof}
For each $n$, the quotient $\Z R/p^n\Z R$ is finite and in
particular an Artin ring (descending chains of ideals become
stationary). Using Lemma 2 b), we see that the image $A_n$ of $I^n$
in $\Z R/p^n\Z R$ is an ideal such that $A_n^2=A_n$.

According to the structure theorem for Artin rings (see e.g.
\cite{AM}, Theorem 8.7), $\Z R/p^n\Z R$ is (uniquely) a finite
direct product $\prod_i B_i$ of local Artin rings $B_i$. Any
idempotent ideal in a local Artin ring $B$ is either $0$ or equal to
$B$  since the maximal ideal in $B$ is nilpotent,
\cite{AM}, 8.2 and 8.4. Therefore the projection of $A_n$ to
any of the components $B_i$ is either $0$ or $B_i$. If we let $e_n$
denote the sum of the identity elements of the $B_i$ in which the
component of $A_n$ is non-zero we get an idempotent $e_n$ in $\Z
R/p^n\Z R$ such that $e_n (\Z R/p^n\Z R) = A_n$ ($e_n$ is a unit
element for $A_n$ and therefore uniquely determined).

Since, by Lemma 2 b), the image of $I^{n+1}$ in $\Z R/p^{n+1}\Z R$
maps surjectively to the image of $I^{n}$ in $\Z R/p^{n}\Z R$ under
the natural map, the sequence $(e_n)$ defines an element $e$ in
$\Z_pR = \mathop{\lim}\limits_{ {\scriptstyle\longleftarrow} } \Z
R/p^n\Z R$. By construction it is an idempotent in $A =
\quad\mathop{\lim}\limits_{ {\scriptstyle\longleftarrow} } A_n=
\quad\mathop{\lim}\limits_{ {\scriptstyle\longleftarrow} } I^n/p^n\Z
R$ such that $ex=x$ for each $x$ in $A$. It follows that $A=e(\Z_p
R)$.

The exact sequence (\ref{xyz}) then shows that the map $(1-e)\Z_pR
\to C(R)$ is a continuous bijective homomorphism between compact
rings and therefore a topological isomorphism.
\end{proof}
\begin{rem} 
The proof shows that $A=e(\Z_pR)$ in the preceding
proposition is a unital ring which is the projective limit of a
system $(A_n)$ of unital rings with unital transition maps. In the case $R = \F_p$ looking at the canonical decomposition of $\Z_p R$ under the action of $\F^{\times}_p$ we see that $e$ has the following explicit description
\[
1 - e = (p-1)^{-1} \sum_{r \in \F^{\times}_p} \omega (r)^{-1} [r] \quad \text{in} \; \Z_p R \; .
\]
Here $\omega$ is the Teichm\"uller character $\omega : \F^{\times}_p \to \Z^{\times}_p$.
\end{rem}


\end{document}